\documentclass[11pt]{amsart}
\usepackage{amsthm}
\usepackage{amssymb}
\usepackage{amsmath}
\usepackage{amscd}

\usepackage{fullpage}
\usepackage[all]{xy}

\newcommand{\intW} {{\mathaccent23{W}}}

\newcommand{\intU} {{\mathaccent23{U}}}

\def \col {\operatorname{\textsf{c}}}

\def \Hom {\operatorname{Hom}}
\def \Desc {\operatorname{Desc}}

\def \inv  {^{-1}}

\def \g {{\mathfrak g}}
\def \fg {{\mathfrak g}}

\def \Z {{\mathbb Z}}
\def \R {{\mathbb R}}
\def \C {{\mathbb C}}

\def \calA {{\mathcal A}}
\def \calB {{\mathcal B}}
\def \calU {{\mathcal U}}

\def \calT {{\mathcal T}}
\def \calU {{\mathcal U}}

\def \eps {\epsilon}

\swapnumbers
\numberwithin{equation}{section}

\newtheorem {lemma}[equation]       {Lemma}

\newtheorem {theorem}[equation]     {Theorem}

\newtheorem {proposition}[equation] {Proposition}

\theoremstyle{definition}
\newtheorem {Definition}[equation]  {Definition}
\newtheorem {definition}[equation]  {Definition}

\theoremstyle{remark}

\newtheorem {remark}[equation]      {Remark}
\newtheorem {notation}[equation]      {Notation}
\newtheorem {empt}[equation]      {}

\newtheorem {Annoying Remark}[equation]      {Annoying Remark}
\newtheorem* {Example*}          {Example}
\newtheorem*{Remark*}            {Remark}

\newcommand{\STM}{\mathsf{STM}}
\newcommand{\STB}{\mathsf{STB}}

\newcommand{\Op}{\textsf{Op}}

\newcommand{\s}{\mathsf{s}}

\begin{document}

\title{Categories of symplectic toric manifolds as Picard stack  torsors}

\author{Eugene Lerman}
\address{Department of Mathematics, The University of Illinois
at Urbana-Champaign, 1409 W.~Green Street, Urbana, IL 61801, U.S.A.}
\email{lerman()math()uiuc()edu}

\thanks{This research is partially supported by the  National
Science Foundation}

\begin{abstract}
We outline a proof that the stack of symplectic toric $G$-manifolds
over a fixed orbit space $W$ is a torsor for the stack of symplectic
toric $G$-principal bundles over $W$.

\end{abstract}


\maketitle

\tableofcontents
\section{Introduction}
\label{sec:intro}
In an influential paper \cite{De} Delzant classified compact
symplectic toric manifolds.  Recently, using some ideas from \cite{L},
Karshon and I \cite{KL} extended the classification to
\emph{non-compact} symplectic toric manifolds: If $(M,\omega)$ is a
symplectic manifold with a completely integrable action of a torus $G$
and an associated moment map $\mu:M\to \fg^*$, then $W=M/G$ is
naturally a manifold with corners \cite{DH}. Furthermore the orbital moment map
$\bar{\mu}:W\to \fg^*$ induced by $\mu$ is locally an embedding that
maps the corners of $W$ to unimodular cones in $\fg^*$. That is,
$\bar{\mu}$ is a \emph{unimodular local embedding}.  

It was easy to classify the isomorphism classes of symplectic toric
manifolds $(M',\omega', \mu':M\to \fg^*)$ with orbit space $W$ and
orbital moment map $\overline{\mu'} = \bar{\mu}$.  These classes are
in bijective correspondence with degree 2-cohomology classes of $W$
with coefficients in $\Z_G\times \R$, where $\Z_G$ denotes the
integral lattice of the torus $G$.

However, showing that given a unimodular local embedding $\psi:W\to
\fg^*$ there exists a symplectic 
toric manifold $(M,\omega, \mu)$ with $M/G= W$ and $\bar{\mu} = \psi$
turned out to be hard.  We dealt with this problem by defining the
category $\STB_\psi (W)$ of symplectic toric $G$-principal bundles
(with corners) over $W$ and constructing a functor 
\[
\col: \STB_\psi (W) \to \STM_\psi (W)
\]
($\col$ for ``collapse'' or ``cut''), where $\STM_\psi (W)$ denotes
the category of symplectic toric manifolds with orbit space $W$ and
orbital moment map $\psi$.  Since $\STB_\psi(W)$ contains the pullback
of the symplectic toric principal $G$-bundle $T^*G\to \fg^*$ by
$\psi:W\to \fg^*$, this proved the $\STM_\psi (W)$ was non-empty as
well.  In this paper I clarify the relation between the two
categories, the categories of toric principal bundles with corners and of
toric manifolds.  The idea is to use the language of stacks.  Observe
that
\begin{itemize}
\item $\STB_\psi (W)$ and $\STM_\psi (W)$ are groupoids.
\item For every open set $U\subset W$ we have groupoids  $\STB_\psi (U)$
and $\STM_\psi (U)$; for an inclusion of two open sets
$V\hookrightarrow U$ we have obvious restriction functors $|_V:
\STB_\psi (U) \to \STB_\psi (V)$ and $|_V: \STM_\psi (U) \to \STM_\psi
(V)$.  Thus we have two (strict!) presheaves $\STM_\psi$ and
$\STB_\psi$ of groupoids on the category $\Op (W)$ of open subsets of
$W$: $U\mapsto \STB_\psi (U)$ and $U\mapsto \STM_\psi (U)$.
\item The two presheaves $\STB_\psi$ and $\STM_\psi$ are actually stacks.
\end{itemize}
The main result of the paper is the following theorem.

\begin{theorem} Let $G$ be a torus and $\psi:W\to \fg^*$ a unimodular
local embedding.  Then
\begin{itemize}
\item The stack $\STB_\psi$ of symplectic toric principal $G$-bundles over
$W$ is a Picard stack.
\item There is an action of $\STB_\psi $ on the stack $\STM_\psi$ of
symplectic toric manifolds over $\psi:W\to \fg^*$ making $\STM_\psi$
into a $\STB_\psi$-torsor.  
\item In particular each choice of a global
object of $\STM_\psi$ defines an isomorphism of stacks $\STB_\psi$ and
$\STM_\psi$.
\end{itemize}
\end{theorem}

\subsection*{Acknowledgments}  I thank Yael Karshon and Anton Malkin 
for our fruitful collaborations without which this paper won't be
possible.

\section{Definitions, notation, conventions}

\noindent
\textbf{Notation and conventions.}
Given a category $\mathsf{A}$ we write $A\in \mathsf{A}$ to indicate
that $A$ is an object of $\mathsf{A}$.  Given a functor
$F:\mathsf{A}\to \mathsf{B}$ we will often describe it by only
indicating what it does on objects.

A \emph{torus} is a compact connected abelian Lie group.  We denote
the Lie algebra of a torus $G$ by $\g$, the dual of the Lie algebra,
$\text{Hom}(\fg,\R)$, by $\g^*$ and the integral lattice, $\ker (\exp
\colon \g \to G)$, by $\Z_G$.    When a torus $G$
acts on a manifold $M$, we denote the action of an element $g \in G$
by $m \mapsto g \cdot m$ and the vector field induced by a Lie algebra
element $\xi \in \g$ by $\xi_M$; by definition, $\xi_M (m) =
\left. \frac{d}{dt}\right|_{t=0} (\exp (t\xi)\cdot m)$.  We write the
canonical pairing between $\fg^*$ and $\fg$ as $\langle \cdot,
\cdot\rangle$.  Our convention for a moment map $\mu \colon M \to
\fg^*$ for a Hamiltonian action of a torus $G$ on a symplectic
manifold $(M, \omega)$ is that it satisfies
\[
d \langle \mu, \xi \rangle = - \omega (\xi_M, \cdot) .
\]
The moment map $\mu$  $G$-invariant; we call the induced map
$\bar{\mu}: M/G \to \fg^*$ the \emph{orbital moment map}.
We say that the triple
$(M, \omega, \mu \colon M \to \fg^*)$ is a \emph{symplectic toric $G$-manifold}
if the action of $G$ on $M$ is effective and if
\[
\dim M = 2 \dim G.
\]
Given a principal $G$-bundle $P$ we write the action of $G$ on $P$ as
a \emph{left} action.  There is no problem with that since $G$ is
abelian.

We denote the positive orthant $\{x\in \R^k \mid x_i\geq 0, 1\leq
i\leq k\}$ by $\R_+^k$.

\begin{definition} Let $\fg^*$ denote the dual of the
Lie algebra of the torus $G$.  A \emph{unimodular cone} in $\g^*$ is a
subset $C$ of $\g^*$ of the form
\begin{equation}\label{def:unimod-cone}
C= \{ \eta \in \fg^* \mid \langle \eta - \eps , v_i \rangle \geq 0
\text{ for all }\,1 \leq i \leq k\},
\end{equation}
where $\eps$ is a point in $\fg^*$ and $\{v_1, \ldots, v_k\}$ is a basis
of the integral lattice of a subtorus $K$ of $G$.
\end{definition}

\begin{empt}
The set $C=\fg^*$ is a unimodular cone with $k=0$,
with $\{v_1,\ldots, v_k\} = \emptyset$, and with $K$ the trivial subgroup
$\{1\}$.
\end{empt}

\begin{Definition} \label{def:unimodular emb}
Let $W$ be a manifold
with corners and $\fg^*$ the dual of the Lie algebra of a torus $G$.
A smooth map $\psi: W \to \g^*$ is a
\emph{unimodular local embedding} if for each point $x$ in $W$ there exists a
unimodular cone $C \subset \g^*$ and open sets $\calT \subset W$
and $\calU \subset \fg^*$ such that $\psi(\calT)= C \cap \calU$
and such that $\psi|_{\calT} \colon \calT \to C \cap \calU$
is a diffeomorphism of manifolds with corners.
\end{Definition}
\noindent
The definition is justified by
\begin{proposition}\label{prop:unimod-emb}
Let $(M, \omega, \mu \colon M\to \fg^*)$ be a symplectic toric $G$-manifold.
Then the orbit space $M/G$ is a manifold with corners,
and the orbital moment map
$\bar{\mu}: M/G\to \fg^*$ is a unimodular local embedding.
\end{proposition}
\begin{proof}
See \cite{KL}.
\end{proof}
\noindent
We now introduce the category of symplectic toric manifolds over a
unimodular local embedding:

\begin{definition}[The category $\STM_\psi (W)$
of \emph{symplectic toric $G$-manifolds over $\psi \colon W\to \fg^*$}]
Let $\psi \colon W \to \fg^*$ be a unimodular local embedding of a
manifold with corners $W$ into the dual of the Lie algebra of a torus
$G$.

An object of the category  $\STM_\psi (W)$  is a pair
$((M,\omega, \mu: M\to \fg^*), \varpi \colon M\to W)$,
where $(M,\omega, \mu \colon M\to \fg^*)$ is a symplectic toric $G$ manifold
and $\varpi \colon M \to W$ is a quotient map 
for the action of $G$ on $M$ with $\mu = \psi \circ \pi$.

A morphism $\varphi$ from $((M, \omega, \mu \colon M \to \fg^*), \varpi
\colon M \to W)$ to $((M', \omega', \mu' \colon M'\to \fg^*), \varpi'
\colon M'\to W)$ is a $G$-equivariant symplectomorphism $\varphi
\colon M \to M'$ such that $\varpi' \circ \varphi = \varpi$.
\end{definition}

\begin{remark}
We may informally write $\varpi \colon M\to W$ or even $M$ for an object of
$\STM_\psi(W)$ and $\varphi \colon M\to M'$ for a
morphism between two objects.
\end{remark}

\begin{definition}[The category $\STB_\psi (W)$
of \emph{symplectic toric $G$-principal bundles over $\psi \colon W\to \fg^*$}]
  An object of $\STB_\psi (W)$ is a principal $G$-bundle $\pi
\colon P \to W$ equipped with a symplectic form $\sigma$ and a moment
map $\mu \colon P \to \fg^*$ such that $\mu = \psi \circ \pi$.  The
morphisms in this category are $G$-equivariant symplectomorphisms that
commute with the maps to $W$.
\end{definition}

\begin{remark} \label{trivial bundle}
The standard lifted action of a torus $G$ on its cotangent
bundle $T^*G$ makes $T^*G$ into a symplectic toric $G$-bundle
over the identity map $id \colon \fg^* \to \fg^*$.
If $\psi \colon W \to \fg^*$ is a unimodular local embedding,
the pullback of $T^*G \to \fg^*$ by $\psi$
gives a symplectic toric $G$ bundle over $\psi \colon W \to \fg^*$.
Thus, for any unimodular local embedding  $\psi \colon W \to \fg^*$,
the category $\STB (\psi \colon W\to \fg^*)$ is non-empty.
\end{remark}

\begin{remark}\label{remark:2.10}
Suppose $\pi: (P,\sigma) \to W$ is a symplectic toric $G$-bundle.
Then it is easy to check that for any connection 1-form $A$ on $P\to
W$ the 2-form
\[
\sigma - d \langle \psi \circ \pi, A\rangle
\]
is basic.  (Recall that $\langle \cdot, \cdot\rangle: \fg^*\times \fg
\to \R$ denotes the canonical pairing.)  Hence any symplectic 2-form
$\sigma$ on the symplectic toric $G$-principal bundle is of the form
\begin{equation} \label{eq:1}
\sigma =  d \langle \psi \circ \pi, A\rangle + \pi^*\beta
\end{equation}
for some connection 1-form $A$ on $P$ and a closed 2-form $\beta$ on
$W$.  Conversely, since $\psi$ is locally an embedding, the 2-form
$d\langle \psi \circ \pi, A\rangle$ is non-degenerate for any
connection 1-form $A$.  Hence all symplectic $G$-invariant forms on
$P$ so that $\psi\circ \pi$ is the corresponding moment map has to be
of the form \eqref{eq:1}.

\end{remark}

\section{A multiplication on the stack of symplectic toric principal
$G$-bundles.}

Fix a torus $G$ and a unimodular local embedding $\psi: W \to \fg^*$
of a manifold with corners $W$ into the dual of the Lie algebra of
$G$.
The following observation must be well known to experts.  If $\pi:P\to
W$, $\pi':P'\to W$ are two principal $G$-bundles over a manifold with
corners $W$, then their fiber product $P\times _WP'$ is a $G\times
G$-principal bundle over $W$.  Dividing out by the action of $G$ given
by
\[
g\cdot (p, p') = (g\cdot p, g\inv \cdot p')
\]
produces a principal $G$-bundle
\[
P\otimes P':= (P\times _WP')/G
\]
over $W$: the induced $G$-action on $P\otimes P'$ is given by
\[
a\cdot [p,p'] = [a\cdot p, p'] = [p,a\cdot p'],
\]
where $[p,p']$ denotes the orbit of $(p,p') \in P\times _W P'$ in
$P\otimes P'$.  Naturally if $P$ and $P'$ are principal bundles over
an open subset $U\subset W$ rather than the whole of $W$, then so is
$P\otimes P'$.  One can show that the  product $\otimes$ turns the gerbe
$BG$ over the site $\Op (W)$ of open subsets of $W$ into a Picard stack.

\begin{proposition} \label{prop:3.1}
For any open subset $U$ of $W$ and any two
symplectic toric principal $G$-bundles $\pi:(P,\sigma)\to U$,
$\pi':(P', \sigma')\to U$ over $U$ the tensor product $P\otimes P' \to
U$ is naturally a symplectic toric principal $G$-bundle.
\end{proposition}

\begin{proof}
We claim that the restriction of the 2-form $\sigma +\sigma' \in
\Omega^2 (P\times P')$ to $P\times _W P'$ is basic with respect to the
action of $G$ given, as above, by $g\cdot (p, p') = (g\cdot p, g\inv
\cdot p')$
and descends to
a non-degenerate form on $P\otimes P'$.  Since this is a local claim,
we may assume that $\psi|_U$ is an embedding or, better yet, that
$U\subset \fg^*$. Then the moment map $\mu$ for the  action of
$G$ on $(P\times P', \sigma +\sigma')$ is
\[
\mu = \pi - \pi'.
\]
Hence
\[
\mu\inv (0)= \{(p,p')\in P\times P'\mid \pi (p) = \pi' (p') \} = P\times _U P'
\]
Since the closed 2-form $(\sigma + \sigma')|_{\mu\inv (0)}$ is
$G$-basic and is degenerate precisely in the directions of $G$-orbits,
it descends to a symplectic form $\sigma\otimes \sigma'$ on $P\otimes
P'$.  It is easy to check that the remaining $G = (G\times
G)/G$-action on $P\otimes P'$ makes $(P\otimes P', \sigma \otimes
\sigma')$ into a symplectic toric $G$-bundle over $U$.
\end{proof}

\begin{notation}
We write
\[
(P,\sigma)\otimes (P',\sigma') := (P\otimes P', \sigma \otimes
\sigma').
\]
\end{notation}

\begin{empt}\label{empt:3.3}
It is easy to extend $\otimes$ to arrows: if $f:(P_1, \sigma_1) \to
(P_1', \sigma_1')$ and $g:(P_2, \sigma_2) \to (P_2', \sigma_2')$ are
two arrows in $\STB_\psi (U)$, then $(f,g): (P_1\times P_2, \sigma_1 +
\sigma_2) \to (P_1'\times P_2', \sigma_1' + \sigma_2')$ descends to
a map $f\otimes g: (P_1\otimes P_2, \sigma_1 \otimes \sigma_2) \to
(P_1'\otimes P_2', \sigma_1' \otimes \sigma_2')$ of symplectic toric
$G$-principal bundles over $U$ for any $U\in \Op(W)$.
\end{empt}
\begin{empt}
It is tedious to check that the multiplication $\otimes$ defined above
gives rise to a structure of the Picard stack on the stack $\STB_\psi$
of symplectic toric principal $G$-bundles over $\Op (W)$.

For example, the canonical isomorphism
$(P\times P', \sigma + \sigma')
\to (P'\times P, \sigma'+\sigma)$ descends to a natural isomorphism
\[
\tau_{P,P'}:(P,\sigma)\otimes (P',\sigma') \to (P',\sigma')\otimes (P,\sigma).
\]
Similarly, the canonical isomorphism $(P_1 \times (P_2\times P_3),
\sigma_1 + (\sigma_2 + \sigma_3)) \to ((P_1 \times P_2)\times P_3,
(\sigma_1 + \sigma_2) + \sigma_3)$ together with a version of
reduction in stages produces a natural isomorphism
\[
\theta_{P_1,P_2,P_3}:
(P_1,\sigma_1)\otimes((P_2,\sigma_2)\otimes(P_3,\sigma_3)) \to
((P_1,\sigma_1)\otimes(P_2,\sigma_2))\otimes(P_3,\sigma_3)
\]
and so on.  The only possibly non-trivial claim is that for any open
set $U\in \Op (W)$ and any symplectic toric principal $G$-bundle
$(P,\sigma) \in \STB_\psi (U)$, the functor $(-)\otimes (P,\sigma):
\STB_\psi (U) \to \STB_\psi (U)$ of multiplication by $(P,\sigma)$ is
an equivalence of categories.  The proof of this claim is the same,
\emph{mutatis mutandis}, as the proof of Theorem~\ref{8-13-7} below.  
We omit it.
\end{empt}

\begin{proposition}\label{8-13-1}
Let $\pi: (P,\sigma) \to U$ be an object of $\STB _\psi (U)$.  A
Lagrangian section $s:U\to P$, if it exists, defines a natural
isomorphism $\s:id_U\Rightarrow (-) \otimes (P,\sigma)$ between the
identity functor $id_U$ on $\STB _\psi (U)$ and the functor $(-)
\otimes (P,\sigma)$ of multiplication by $(P,\sigma)$.
\end{proposition}

\begin{proof}
Let $f:(Q_1, \tau_1)\to (Q_2,\tau_2)$ be an arrow in $\STB_\psi (U)$.
Consider the map $\s_{Q_i}: (Q_i, \tau_i) \to (Q_i, \tau_i) \otimes
(P,\sigma)$, $i=1,2$, is given by
\[
\s_{Q_i} (q) = [q, s (\pi_i (q))]
\]
where, as before, $[q, p] \in (Q_i\times _U P)/G$ the orbit of $(q,p)
\in Q_i\times _U P$ and $\pi_i:Q_i \to U$ denotes the projection. It
is clearly a map of principal $G$-bundles.  Moreover, since $s$ is
Lagrangian, $\s_{Q_i}$ is symplectic.  Recall that  $(f\otimes
id_P)([q,p]) = [f(q), p]$ by definition.  Hence
\[
\s_{Q_2} \circ f (q) = [f(q), s(\pi_2 (f(q))].
\]
Now, $\pi_2 (f(q)) = \pi _1 (q)$.  Therefore
\[
(f\otimes id_P)\circ s_{Q_1} (q) =
[f(q), s (\pi_1 (q))] = \s_{Q_2} \circ f (q).
\]
\end{proof}

\begin{remark}\label{NB}
Note that if $f:Q_1 \to Q_2$ is a map of principal $G$-bundles (not
necessarily symplectic!), then $f\otimes id_P: Q_1\otimes P \to Q_2
\otimes P$ is a map of principal bundles, and we still have
\begin{equation}\label{eq:*}
(f\otimes id_P)\circ s_{Q_1}  = \s_{Q_2} \circ f.
\end{equation}
\end{remark}

\begin{empt}
Note the global neutral object of $(\STB_\psi, \otimes, \tau, \theta)$
is the pullback $W\times _{\fg^*} T^*G$ of the canonical symplectic
toric principal $G$-bundle $T^*G \to \fg^*$ by the map $\psi:W\to
\fg^*$.
\end{empt}

We end the section with two technical observations that we will need later on.
\begin{lemma} \label{8-13-2}
If $V, U$ are two open sets in $W$ and $V\subset U$ is dense in $U$,
then the restriction functors $|_V: \STM_\psi (U)\to \STM_\psi (V)$ and
$|_V: \STB_\psi (U)\to \STB_\psi (V)$ are faithful.
\end{lemma}
\begin{proof}
If $V\subset U$ is dense, then for any $M\in \STM_\psi (U)$ the
restriction $M_V$ is dense in $M$.  Therefore, for any two arrows
$f,g:M\to M'$ in $\STM_\psi (U)$
\[
f|_{M|_V} = g|_{M|_V} \quad \textrm{implies that}\quad f =g.
\]
The proof for symplectic toric principal bundles is the same.
\end{proof}

\begin{lemma}\label{8-13-3}
Suppose an open subset $U$ of $ W$ is contractible.  Then
\begin{enumerate}
\item \label{8-13-3.i} All objects of $\STB_\psi (U)$ are isomorphic.
Consequently the stack $\STB_\psi$ is a gerbe.
\item \label{8-13-3.ii}Any object of $\STB_\psi (U)$ has a Lagrangian section.
\end{enumerate}
\end{lemma}
\begin{proof}
Since $U$ is contractible all principal $G$-bundles over $U$ are
trivial, hence isomorphic.  Therefore, given two objects
$\pi:(P,\sigma)\to U$, $\pi':(P', \sigma')\to U$ of $\STB_\psi (U)$
we may assume $P=P'$. By remark~\ref{remark:2.10}    $\sigma =
d\langle \psi \circ \pi, A\rangle + \pi^* \beta $ and $\sigma' =
d\langle \psi \circ \pi, A\rangle + \pi^* \beta '$ for a connection
1-form $A$ on $P$ and closed 2-forms $\beta, \beta'$ on $U$.  By
Poincare lemma $\beta -\beta' = d\gamma$ for some 1-form $\gamma$ on
$U$.  Hence
\[
\sigma' = \sigma +\pi^* (d\gamma).
\]
By \cite[Lemma~5.10]{KL} there exists a gauge transformation
$\phi:P\to P$ with $\phi^*\sigma' = \sigma$.  This proves
\ref{8-13-3.i}.

By \ref{8-13-3.i} we may assume $(P,\sigma) = (U\times G, d\langle
\psi \circ \pi, g\inv dg\rangle)$ where $g\inv dg $ denotes the
Maurer-Cartan form on the torus $G$.  Then $s:U\to U\times G$ given
by $s(x) = (x, 1)$ is a desired Lagrangian section.
\end{proof}

\section{An action of the stack of symplectic toric principal
bundles on the symplectic toric manifolds}

Next we construct an action of $\STB_\psi$ on $\STM_\psi$.  We first
observe that

\begin{lemma}
For any symplectic toric principal $G$-bundle $(\pi: (P,\sigma) \to U)\in
\STB_\psi (U)$ and any symplectic toric $G$-manifold
$(\varpi: (M,\omega) \to U)\in
\STM_\psi (U)$ over $U$, the fiber product
\[
  P\times _U M = \{(p,m)\mid \pi(p) = \varpi (m)\}
\]
is a manifold with a free action of $G$ given by $g\cdot (p,m) =
(g\cdot p, g\inv \cdot m)$.

Moreover, the form $\sigma + \omega|_{P\times _U M} $ descends to a $G$-invariant symplectic form $\sigma*\omega$ on the quotient
\[
P*M:=  (P\times _U M)/G
\]
making $(P*M, \sigma*\omega) \to U$ into a symplectic $G$-toric
manifold over $U$.
\end{lemma}
\begin{proof}
This is similar to the proof of \ref{prop:3.1}.  The difference is
that it is not completely obvious that the fiber product $P\times _U
M$ is a manifold (without corners), since $P$ is a manifold with corners.

It is no loss of generality to assume that $G= \R^n/\Z^n$, $U\subset
\fg^* \simeq \R^n$ is of the form $U= \R^k_+ \times \R^\ell$ for some $k$,
$\ell$ with $k+\ell = n$, $P= U\times G$ with $\pi:U\times G\to U$
given by $\pi (p, q) = p$ and $M= \C^k\times \R^\ell \times
\R^\ell/\Z^\ell$ with $\varpi:M\to U$ given by $\varpi (z,\eta, \theta) =
(|z_1|^2, \ldots, |z_k|^2, \eta_1, \ldots, \eta_\ell)$.  Then
\[
 P\times _U M = \{(p,q, z,\eta, \theta)
\in U\times G \times \C^k\times \R^\ell \times \R^\ell/\Z^\ell
\mid
(p_1,\ldots, p_n) = (|z_1|^2, \ldots, |z_k|^2, \eta_1, \ldots, \eta_\ell)\}
\]
is the image of the map $\nu: G\times M \to P\times M = U\times G
\times M$ given by
\[
\nu (g, z, \eta, \theta) = \left((|z_1|^2, \ldots, |z_k|^2,
\eta_1, \ldots, \eta_\ell), g, (z, \eta, \theta)\right).
\]
Since for any coordinate function $f:U\times G \times M \to \R$ the
composite $f\circ \nu$ is smooth, $\nu$ is smooth.  And, in fact,
$P\times _U M$ is the product of the graph of a smooth map
\[
\lambda : \C^k \to \R^k_+,\quad \lambda (z) = (|z_1|^2, \ldots, |z_k|^2)
\]
with $G\times \R^\ell \times (\R^\ell/\Z^\ell)$.  Hence $\nu$ is an
embedding of a manifold without corners into a manifold with corners.
This proves that $P\times _U M$ is a manifold without corners.

The rest of the argument is the same as in Proposition~\ref{prop:3.1}.
\end{proof}
\begin{empt}
It is easy to see that if $\varpi: (M,\omega) \to W$ is a symplectic
toric manifold over $\psi:W\to \fg^*$ and $\pi:(P,\sigma) \to U$ is
a symplectic principal $G$ bundle over an open subset
$U\hookrightarrow W$ of $W$, then the fiber product $P\times _W M$
is a manifold and $P*M :=(P\times _W M)/G$ is naturally a symplectic
toric manifold over $U$.   Indeed it is $(P,\sigma) * (
(M,\omega)|_U)$.  We will denote it, by a slight abuse of notation,
by $(P,\sigma) * (M,\omega) = (P*M, \sigma*\omega)$.
\end{empt}
\begin{empt}
The map $*$ extends to arrows.  The pairs of arrows $f:(P,\sigma) \to
(P', \sigma')$ and $g:(M,\omega)\to (M',\omega')$ in $\STB_\psi(U)$
and $\STM_\psi (U)$ respectively define a $G\times G$ equivariant
symplectmorphism $(f,g):(P\times M, \sigma + \omega) \to (P'\times M',
\sigma' +\omega')$ which maps $P\times _U M$ to $P'\times _U M'$.
Hence it induces an arrow $f*g: (P,\sigma)*(M,\omega) \to (P',
\sigma')*(M',\omega')$ in $\STB_\psi (U)$.

We leave it to the reader to check that
\[
*: \STB_\psi (U) \times \STM_\psi (U) \to \STM_\psi (U)
\]
is a functor.  It is also not hard to check that $*$ commutes with
restrictions: for any $V, U $ open subsets of $W$ with $V\subset U$
the diagram
 \begin{equation}\label{eq:comm-w-restrictions}
\xymatrix{
 \STB_\psi (U) \times \STM_\psi (U)  \ar[r]^{\quad\quad   * } \ar[d]^{|_V}
&  \STM_\psi (U) \ar[d]^{|_V}  \\
\STB_\psi (V) \times \STM_\psi (V)  \ar[r]^{\quad\quad* }
&  \STM_\psi (V) }
\end{equation}
commutes.  Thus
\[
*: \STB_\psi  \times \STM_\psi  \to \STM_\psi
\]
is a strict map of presheaves of groupoids.
\end{empt}

We will not check all the details necessary to show that $*$ defines
an action of the Picard stack $\STB_\psi$ on the stack $\STM_\psi$.
Our goal is to prove

\begin{theorem}\label{thm:main}
The action $*$ of the Picard stack $\STB_\psi$ of symplectic toric
principal $G$-bundles over a unimodular local embedding $\psi:W\to
\fg^*$ on the stack $\STM_\psi$ of symplectic toric manifold over
$\psi$ makes $\STM_\psi$ into an $\STB_\psi$-torsor.  That is, the
functor
\[
a: \STB_\psi \times \STM_\psi \to \STM_\psi \times \STM_\psi, \quad
a((P, \sigma), (M,\omega)):= ((M,\omega), (P,\sigma)*
(M,\omega))
\]
is an isomorphism of stacks.
\end{theorem}

Theorem~\ref{thm:main} is an easy consequence of a seemingly weaker
result, which is of an independent interest:
\begin{theorem}\label{8-13-7}
For any symplectic toric manifold $\varpi:(M,\omega) \to W$ over a
unimodular local embedding $\psi:W\to \fg^*$ the functor
\[
F_M: \STB_\psi \to \STM_\psi,\quad F_M( (P,\sigma)):= (P,\sigma)
*(M,\omega)
\]
is an isomorphism of stacks.
\end{theorem}

\begin{proof}[Proof of Theorem~\ref{thm:main} assuming \ref{8-13-7}]
To show that $a$ is essentially surjective, given $(M,\omega),
(M',\omega')\in \STM_\psi (U)$ we need to find $(P,\sigma) \in \STB_\psi (U)$ so
that
\begin{equation}\label{eq:thm:main}
(P,\sigma)*(M,\omega) \quad \textrm{is isomorphic to} \quad (M',\omega').
\end{equation}
By \ref{8-13-7}, the functor $F_M (U): \STB_\psi (U) \to \STM(U)$,
$(P,\sigma) \mapsto (P,\sigma)*(M,\omega)$ is essentially surjective.
Hence for any $(M',\omega') \in \STM_\psi (U)$ there is $(P,\sigma)\in
\STB_\psi (U)$ so that \eqref{eq:thm:main} holds.

Next suppose $(f,h), (g,k): ((P,\sigma), (M,\omega)) \to
((P',\sigma'), (M',\omega'))$ are two arrows in $\STB_\psi (U)\times
\STM_\psi (U)$ with
\[
a (f,h) = a (g,k).
\]
Then
\[
h=k \quad \textrm{and} \quad f*h = g*k.
\]
Hence
\[
g*id_M = (g*k)\circ (id_{P'}* h\inv) = (f*h)\circ  (id_{P'}* h\inv)=
f*id_M.
\]
By \ref{8-13-7} again, $g=f$.  Hence $a$ is faithful.

Finally we argue that $a$ is full.  Suppose we have an arrow $(h,f):
((M,\omega), (P*M,\sigma*\omega) ) \to ((M',\omega'), (P'*M',
\sigma'*\omega'))$ in $\STM_\psi (U) \times \STM_\psi (U)$. We want to
find an arrow $\tilde{f}:(P,\sigma)\to (P',\sigma')$ in $\STB_\psi
(U)$ with
\[
(h,f) = a (\tilde{f}, h) \equiv (h,\tilde{f}*h).
\]
Consider $(id_{P'} *h\inv)\circ f: P*M\to P'*M$.  By \ref{8-13-7},
there is an arrow $\tilde{f}:(P,\sigma) \to (P', \sigma')$ with
\[
\tilde{f}*id_M = (id_{P'} *h\inv)\circ f.
\]
Then
\[
f = (id_{P'} *h\inv)\inv \circ (\tilde{f}*id_M)
= (id_{P'} *h) \circ (\tilde{f}*id_M) = \tilde{f}*h.
\]
Thus $a$ is full.
\end{proof}

As a preparation for our proof of \ref{8-13-7} we prove a number of
lemmas.  But first,

\begin{notation}
Given a manifold with corners $W$ we denote its interior by  $\intW$.  For any open subset $U$ of $W$ we set
\[
\intU = U\cap \intW.
\]
\end{notation}

\begin{lemma}\label{8-13-4}
Suppose $U\in \Op (W)$ is such that $\intU$ is contractible.  Then the functor
\[
F_M (U): \STB_\psi (U) \to \STM_\psi (U),
(f:(P,\sigma) \to (P',\sigma')) \mapsto f*id_M
\]
is faithful.
\end{lemma}
\begin{proof}
Since $\intU$ is contractible, the symplectic toric principal
$G$-bundle $M|_\intU\to \intU$ has a Lagrangian section by
\ref{8-13-3}.  By \ref{8-13-1},
$F_M (\intU) = F_M(U)|_\intU:\STB_\psi (\intU) \to \STB_\psi (\intU) =
\STM_\psi (\intU)$ is isomorphic to the identity functor, hence is
faithful.  Since the diagram
\[
\xymatrix{ \STB_\psi (U)   \ar[r]^{F_M (U) }
\ar[d]^{|_\intU}
&  \STM_\psi (U) \ar[d]^{|_\intU}  \\
\STB_\psi (\intU)  \ar[r]^{F_M (U)|_\intU }
&  \STM_\psi (\intU) }
\]
commutes and since the functors $|_\intU$ are faithful by \ref{8-13-2},
the functor $F_M(U)$ is faithful as well.
\end{proof}

\begin{lemma}\label{8-13-5}
Let $U\in \Op(W)$ be contractible.  Suppose further that $\intU$ is
contractible as well.  Then the functor
\[
F_M (U):\STB_\psi (U)\to \STM_\psi (U)
\]
is full.
\end{lemma}
\begin{proof} \emph{Step 1.}
Let $f:P*M\to P*M$ be an arrow in $\STM_\psi (U)$ (to streamline our
notation we are no longer explicitly keeping track of the symplectic
forms). We will argue that there is  an arrow $\tilde{f}:P\to
P'$ in $\STB_\psi (U)$ with $\tilde{f}*id_M = f$.  By
\cite[Theorem~3.1]{HS} there is a smooth map $h:U\to G$ with
\[
f(x) = h(\varpi(x))\cdot x
\]
for all $x\in P*M$ (here, again, $\varpi:P*M\to U$ denotes the orbit
map).  Define a gauge transformation $\tilde{f}:P\to P$ of the
principal $G$ bundle $\pi:P\to U$ by
\[
\tilde{f}(p):= h(\pi (p) \cdot p.
\]
Then for any point $[p,m] \in P*M$ we have
\[
(\tilde{f} *id_M )[p,m] = [h(\pi (p))\cdot p, m] = h(\pi (p))\cdot [p,m]
= h(\varpi ([p,m]))\cdot [p,m] = f([p,m]).
\]
It remains to check that $\tilde{f}$ actually preserves the symplectic
form on $P$.  Since $P|_\intU$ is dense in $P$, it's enough to check
that $\tilde{f}|_\intU := \tilde{f}|_{P|_\intU}$ is symplectic.  Since
$\intU$ is contractible by assumption, the bundle $M|_\intU\to \intU$
has a Lagrangian section.  By \ref{8-13-1} and \ref{NB} we have a
symplectic $G$-equivariant diffeomorphism $\alpha :P|_\intU \to
(P*M)|_\intU = P|_\intU*M|_\intU$ so that the diagram
\[
\xymatrix{
 P|_\intU  \ar[r]^{\alpha} \ar[d]^{\tilde{f}}
&  P*M\,|_\intU \ar[d]^{f= \tilde{f}*id}  \\
P|_\intU  \ar[r]^{\alpha } &  P*M\,|_\intU }
\]
commutes (all maps are maps of principal $G$-bundles).  But we also
know that $\alpha$ and ${f}$ are \emph{symplectic}.  Hence
$\tilde{f}|_\intU$ is symplectic.  Therefore $\tilde{f}$ is
symplectic.

\emph{Step 2.} We now argue that $F_M (U):\Hom (P,P') \to \Hom (P*M, P'*M)$ is onto for any pair of symplectic toric principal $G$-bundles $P,P'$ over $U$.  Since $U$ is contractible, $P$ and $P'$ are isomorphic by \ref{8-13-2}.  Let $\psi:P\to P'$ denote this isomorphism.  Then for any $\phi\in \Hom (P*M, P'*M)$, the map $(\psi *id_M)\circ \phi $ is in $\Hom(P*M,P*M)$.  By \emph{Step 1},
\[
(\psi *id_M)\circ \phi  = \tau*id_M
\]
for some $\tau\in \Hom(P,P)$.   Hence
\[
\phi = (\psi\inv \circ \tau)*id_M,
\]
and we are done.
\end{proof}

\begin{notation}
Given an open cover $\{U_i\}$ we write $U_{ij}$ for the double
intersection $U_i\cap U_j$ and $U_{ijk}$ for $U_i\cap U_j\cap U_k$.
\end{notation}
\begin{lemma}\label{8-13-6}
Suppose $\calA$ and $\calB$ are two stacks over the site $\Op(W)$ of
open subsets of a manifold with corners $W$ thought of as strict
presheaves of groupoids and $F:\calA\to \calB$ a map of strict
presheaves, so that for any pair $U'\hookrightarrow U$ of open sets
the diagram
\[
\xymatrix{
\calA (U)   \ar[r]^{F(U) }
\ar[d]^{|_{U'}}
&  \calB (U) \ar[d]^{|_{U'}}  \\
\calA (U')  \ar[r]^{F (U') }
&  \calB (U')}
\]
commutes.  If for any $U\in \Op(W)$ there is a cover $\{U_i\}$ so that
$F(U_i):\calA(U_i)\to \calB (U_i)$, $F(U_{ij}): \calA(U_{ij})\to
\calB (U_{ij} )$ and $F(U_{ijk} ): \calA(U_{ijk})\to \calB (U_{ijk} )$
are all equivalences of categories for all $i,j,k$ then $F(U)$ is also
an equivalence of categories.  Hence $F:\calA\to \calB$ is an
isomorphism of stacks.
\end{lemma}

\begin{proof}
Since $\calA$ is a stack, for any $U\in \Op(U)$ and any cover
$\{U_i\}$ of $U$ the category $\calA (U)$ is equivalent to the descent
category $\Desc (\{U_i\}, \calA)$ defined by the cover $\{U_i\}$.  The
conditions on $F$ guarantee that the induced functor
\[
F: \Desc (\{U_i\}, \calA) \to \Desc (\{U_i\}, \calB)
\]
between the descent categories is an equivalence.  It follows that
$F(U)$ is an equivalence of categories.
\end{proof}

\begin{proof}[Proof of \ref{8-13-7}]
We need to show that for any open subset $U$ of $W$, the functor
\[
F_M (U): \STB_\psi (U) \to \STM _\psi (U)
\]
is an equivalence of categories.

If $U$ is contractible the cohomology $H^2
(U, \Z_G \times \R)$ is 0. Hence, by \cite[Theorem 1.8]{KL}, all objects
of $\STM_\psi (U)$ are isomorphic.  Therefore $F_M (U)$ is essentially
surjective.  Furthermore if $\intU$ is contractible, then $F_M (U)$ is
faithful by \ref{8-13-4}.  By \ref{8-13-5}, $F_M (U)$ is full.

We have seen \eqref{eq:comm-w-restrictions} that $F_M$ strictly
commutes with restrictions.  Now choose an open cover $\{U_i\}$ of $W$
s so that
\begin{itemize}
\item all sets $U_i$, $U_{ij} := U_i \cap U_j$ and
$U_{ijk} := U_i\cap U_j\cap U_k$ are contractible and
\item their interiors $\intU_i$, $\intU_{ij}$ and
$\intU_{ijk}$ are contractible as well.
\end{itemize}
This can be achieved, for example, by choosing a triangulation
\cite{Go,J} of the manifold with corners $W$ compatible with its
stratification into manifolds and using open stars of the
triangulation as the elements of the cover.  By \ref{8-13-6} the
functor $F_M$ is an isomorphism of stacks.
\end{proof}


\end{document}